\documentclass[11pt]{scrartcl}

\usepackage{amsmath}    
\usepackage{amssymb}   	
\usepackage{amsthm}     
\usepackage{graphicx}    

\usepackage{textcomp}    
\usepackage[T1]{fontenc} 
\usepackage{marvosym}  

\usepackage{authblk} 	
\usepackage[usenames]{xcolor} 
\usepackage{lastpage} 	

\usepackage[utf8]{inputenc}									
\usepackage{indentfirst}
\usepackage{amsfonts,amscd}	
\usepackage{empheq}
\usepackage{latexsym}		
\usepackage{graphicx,graphics,epsfig}								
\usepackage{graphpap}										
\usepackage{float}
\usepackage{xcolor}
\usepackage{makeidx}      
\usepackage{appendix}	  
\usepackage{multicol}										
\usepackage{color}

\usepackage[normalem]{ulem}
\usepackage{changepage}
\usepackage{hyperref} 
\hypersetup{
	colorlinks=true,
	linkcolor=blue,         
   citecolor=blue,          
   urlcolor=blue          
}

\numberwithin{equation}{section}
\numberwithin{table}{section}
\numberwithin{figure}{section}

\newtheorem{theorem}{Theorem}[section]

\newtheorem{lemma}[theorem]{Lemma}

\newtheorem{conjecture}[theorem]{Conjecture}

\theoremstyle{definition}

\theoremstyle{definition}

\theoremstyle{definition}
\newtheorem{example}[theorem]{Example}

\usepackage[nameinlink,capitalize,noabbrev]{cleveref}

\usepackage{thmtools}
\usepackage{thm-restate}

\newtheorem{question}[theorem]{Question}

\newcommand{\ZZ}{\mathbb Z}

\newcommand{\ie}{i.\,e.,}

\author{
    Anna M.\ Limbach%
    \thanks{Institute of Computer Science, Czech Academy of Sciences, Pod Vod\'{a}renskou v\v{e}\v{z}\'{i} 271/2, 182 00 Praha 8, Czech Republic, email: limbach@cs.cas.cz}\hspace{1.7ex}and 
    Martin Winter%
    \thanks{Mathematics Institute, Technische Universit\"at Berlin, Stra\ss{}e des 17.\ Juni 135, 10623 Berlin, Germany \newline  email: winter@math.tu-berlin.de. Martin Winter is supported as Dirichlet Fellow by the Berlin~Mathematics Research Center MATH+ and the Berlin Mathematical School.}
}

\title{When do graph covers preserve the clique dynamics of infinite graphs?}

\begin{document}

\date{\vspace{-2em}}
\maketitle

\begin{abstract}
    \noindent We investigate for which classes of (potentially infinite) graphs the clique dynamics is cover stable, \ie\ when clique convergence/divergence is preserved under triangular covering maps.

We first present an instructive counterexample:
a clique convergent graph which covers a clique divergent graph and which is covered by a clique divergent graph.
%
%
%
Based on this we then focus on local conditions (\ie\ conditions on the neighbourhoods of vertices) and show that the following are sufficient to imply cover stability:
local girth $\ge 7$ and local minimum degree $\ge 2$; being locally cyclic and of minimum degree $\ge 6$.


\end{abstract}



\section{Introduction} 

Given a possibly infinite but locally finite graph $G$ (\ie\  every vertex degree is finite), its \emph{clique graph} $kG$ is the intersection graph of its maximal complete subgraphs, called \emph{cliques}. This means the set of cliques of $G$ is the vertex set of $kG$ and two vertices of $kG$ are adjacent if and only if they intersect non-trivially as cliques. The operator $k$ that maps each graph to its clique graph is called the \emph{clique graph operator}. As the clique graph of a locally finite graph is locally finite, we can apply the operator iteratively to generate the sequence of iterated clique graphs $k^0G=G$, $kG$, $k^2(G)=k(kG)$, $k^3G$, ... . In the following, we are interested in the dynamics of this operator.

If for each pair of non-negative integers $m$ and $n$ the graphs $k^mG$ and $k^nG$ are non-isomorphic, $G$ is called \emph{clique divergent}; otherwise, \emph{clique convergent}. As $k^mG\cong k^nG$ implies $k^{m+1}G\cong k^{n+1}G$, for a clique convergent graph the sequence of iterated clique graphs starts with a finite pre-periodic phase before it starts to cycle on a finite number of graphs.

We want to describe the relation between clique dynamics and triangular covering maps for graphs. Hereby the graphs may be finite or infinite but locally finite. A graph homomorphism is called a \emph{triangular covering map} if its restriction to the closed neighbourhood of each vertex is a bijection or, equivalently, if two different preimages of the same vertex have always distance more than three. In this case, the source graph is called a \emph{triangular cover} of the target graph. The \emph{universal triangular cover} of a graph $G$ is the unique connected triangular cover of $G$ which covers every triangular cover of $G$.
Existence and uniqueness up to isomorphism are shown in \cite{BAUMEISTER2022112873}.

In \cite{larrion2000locally}, it was proven that for each triangular covering map between two finite graphs the source is clique convergent if and only if the target is. 

Their proof consists of three arguments. Firstly, a finite graph is clique divergent if and only if the number of vertices of its iterated clique graphs diverges (using the pigeonhole principle); secondly, 
the clique graph of a triangular cover of a graph $G$ is a triangular cover of $kG$; and thirdly, for a triangular covering map, the size of the fibre of any vertex has the same size and this number is also constant across the iterated clique graphs and their respective covers. Unfortunately, those arguments do not sensibly extend to settings in which not both the original graph and its cover are finite.

In fact, neither direction of the theorem generalizes to the non-finite setting. In \cref{sec:counterexample} we construct a clique convergent graph which both has a clique divergent triangular cover and is the triangular cover of a clique divergent graph. 

Still, the following theorem regarding universal covers remains true: 

\begin{theorem}[\cite{LIMBACH2024114144}]
     If a locally finite graph $G$ is clique convergent, so is its universal triangular cover.
\end{theorem}

This follows from two facts: the clique graph of a universal triangular cover is a universal triangular cover of the clique graph; and universal triangular covers are unique up to isomorphism. For a more in-depth discussion see  \cite{LIMBACH2024114144}.

Since in general the clique dynamics of a graph and of its triangular covers can be very different, a next step would be to look at restricted classes of graphs. 
We say that the clique dynamics of a graph class is \emph{cover stable} if clique convergence/divergence of a graph from the class implies clique convergence/divergence for each triangular cover.
In this article we specifically look at classes defined by local properties, \ie\ properties of vertex neighbourhoods. This is natural as those classes are closed under taking triangular covers.

First, we mention the class of locally cyclic graphs with minimum degree $6$ or higher. A graph is called \emph{locally cyclic} if all open neighbourhoods are cycles. In \cite{BAUMEISTER2022112873} and \cite{LIMBACH2024114144} it was shown that those graphs are clique convergent if and only if their universal triangular covers are, and the following characterization was given:

\begin{theorem}[\cite{LIMBACH2024114144}]\label{thm:loccyc}
    Let $(\Delta_k)_{k\in\mathbb{N}_0}$ be the family of graphs depicted in \cref{fig_pyramids}. 
    
    \begin{figure}[htbp]
	\centering
	\includegraphics[width=0.5\textwidth]{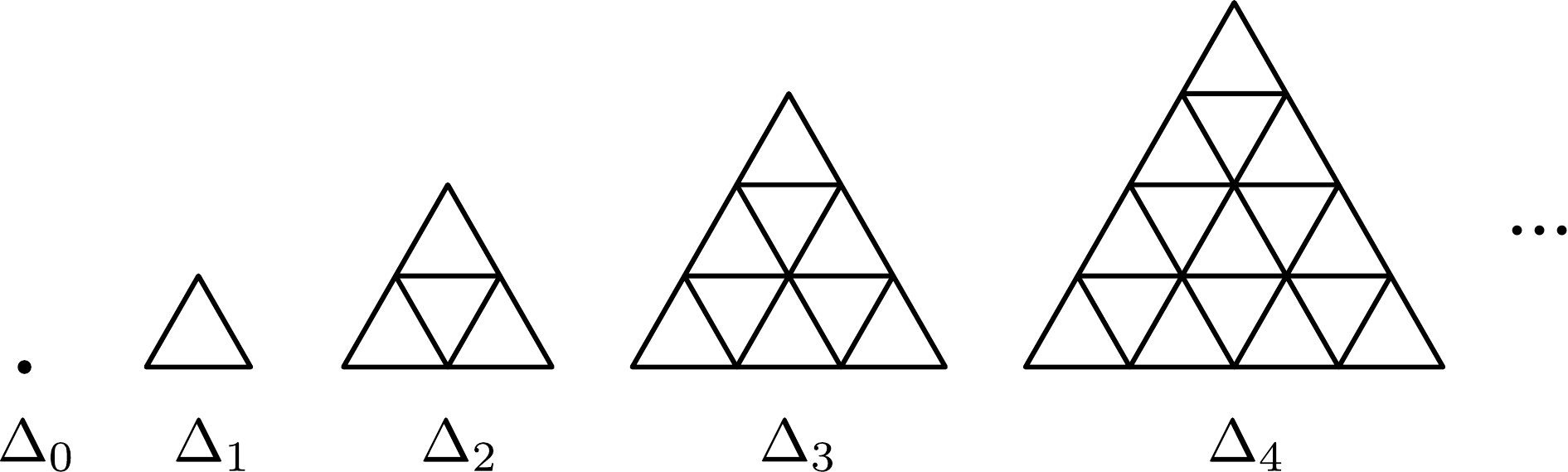}
	\caption{The graphs $\Delta_0$ to $\Delta_4$.}
	\label{fig_pyramids}
\end{figure}
    A locally cyclic graph $G$ with minimum degree $\delta\geq 6$ is clique divergent if and only if for each $k\in\mathbb{N}_0$, the graph $\Delta_k$ can be embedded into the universal triangular cover of $G$.
\end{theorem}

By definition, the universal triangular cover is the same for a graph and any of its triangular covers. Thus, by \cref{thm:loccyc}, the clique dynamics is cover stable on the class of locally cyclic graphs with minimum degree at least $6$. 
Remarkably, even  though the class is defined by a local property, the characterization of clique dynamics depends on the global structure of the graphs in this case.

We move on to the class of graphs whose local girth (\ie\ the minimum girth of the open vertex neighbourhoods) is at least 7. 
In \cite{LARRION2002123}, it was shown that all finite graphs in this class are clique convergent.
The proof uses that for any such graph $G$, its clique graph $kG$ is \emph{clique-Helly}, that is, any set of pairwise intersecting cliques of $kG$ has a non-empty total intersection.
It is furthermore used that if $H$ is clique-Helly, then all iterated clique graphs of $H$ are themselves clique-Helly, and that $k^2H$ is an induced subgraph of $H$ \cite{escalante1973iterierte}. 
Both proofs generalize directly to the infinite case. 
Since for a finite graph $G$ there cannot be infinite chains of proper induced subgraphs, $H=kG$ (and thus $G$) are clique convergent with period length $1$ or $2$. 
For infinite graphs, however, those infinite chains can exist, and a different argument is necessary. 
In \cref{sec:cliquehelly}, we first construct clique-Helly graphs that are clique convergent of arbitrary period length, as well as clique divergent examples. We also note that clique-Hellyness is cover stable. 
The constructed graphs have infinite local girth, which shows that not all infinite graphs of local girth at least $7$ are clique convergent. 
We then show that imposing stricter constraints salvages the result (proven in \cref{sec:proofthmlocgirth}):

\begin{restatable}{theorem}{locgirth}\label{thm:locgirth}
    A locally finite graph with local girth at least $7$ and local minimum degree at least $2$ (\ie\ all open neighbourhoods have minimum degree at least $2$) is clique convergent.  
\end{restatable}

As local properties are cover stable, this class is closed under triangular covers and trivially, the clique dynamics is cover stable on it.
The following question currently remains open:

\begin{question}
    Is the clique dynamics of graphs of local girth $\ge 6$ and local minimum degree $\ge 2$ cover stable? 
\end{question}

\section{A Counterexample to Cover Stability of Clique Dynamics}\label{sec:counterexample}

In this section we construct a graph $T$ that is itself clique convergent, yet has a clique divergent cover as well as covers a clique divergent graph. This shows that in general the clique dynamics needs neither be upwards nor downwards cover stable.

\textit{Constructing $T$:}
Start with the infinite 3-regular tree (\cref{fig:TCG} (a)).
Choose a 1-factor $E$ in $T$ (\ie\ a set of pair-wise non-adjacent edges that cover all vertices of $T$). The complement of $E$ consists of infinitely many disjoint two-way infinite paths, any two of which are connected by at most one edge from $E$ (\cref{fig:TCG} (b)).
On each such path choose a $\mathbb Z$-edge labelling in the natural way (\ie\ choose an origin and a direction on the path and label the edges consecutively with $...,-2,-1,0,+1,+2,...$) subject to the following constraint: whenever two path $P_1$ and $P_2$ are connected by an edge $e\in E$, then the vertices $P_1\cap e$ and $P_2\cap e$ are incident to edges with the same labels (\cref{fig:TCG} (c)).
Next, subdivide each labelled edge $f$ with a new vertex and attach to this vertex a path of length $\max\{0,n\}$, where $n$ is the label on $f$ (\cref{fig:TCG} (d)).
Finally, replace each edge in $E$ by two parallel paths of length four (forming an 8-cycle) (\cref{fig:TCG} (e)).
The resulting graph is $T$.

\begin{figure}
    \centering
    (a)\quad\includegraphics[width=0.2\linewidth]{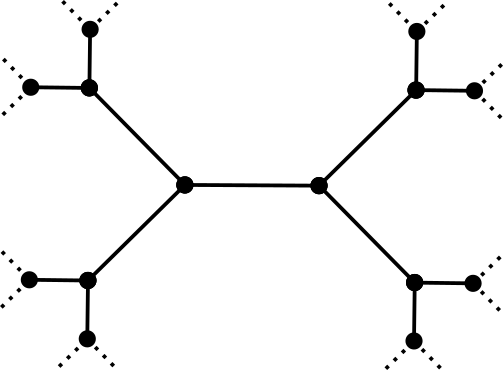}\qquad
    (b)\quad\includegraphics[width=0.2\linewidth]{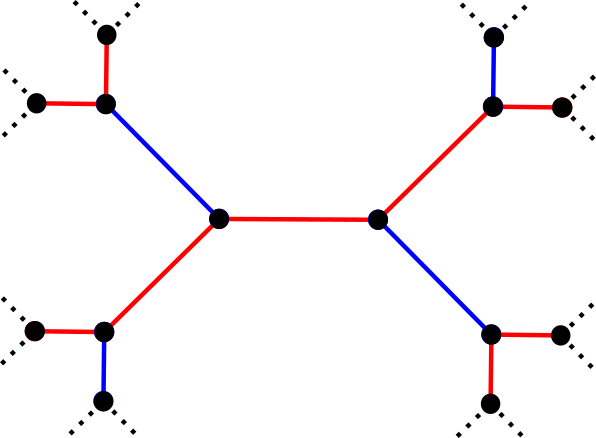}\qquad
    (c)\quad\includegraphics[width=0.24\linewidth]{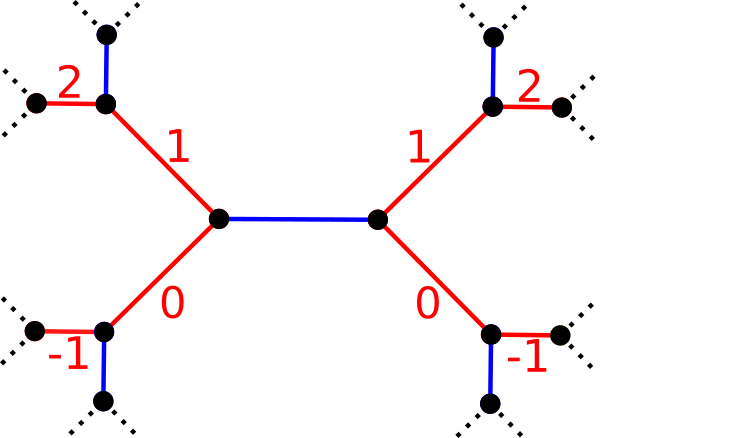}\\[4mm]
    (d)\quad\includegraphics[width=0.3\linewidth]{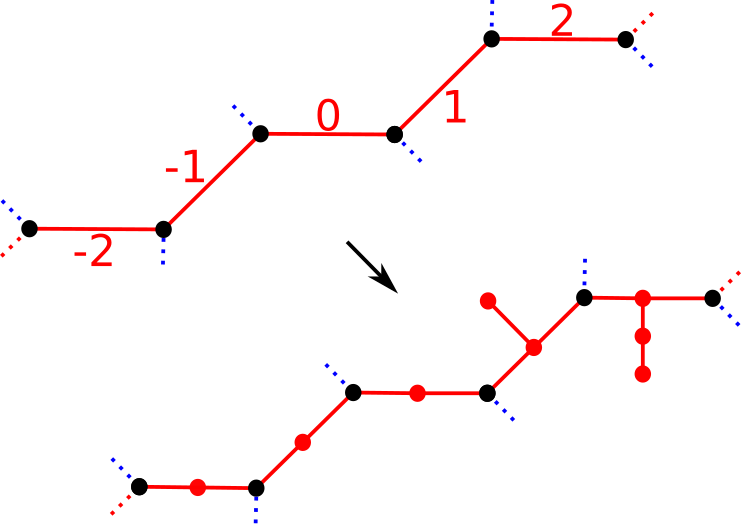}\qquad
    (e)\quad\includegraphics[width=0.3\linewidth]{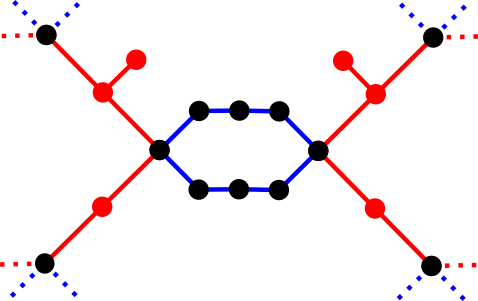}
    \caption{(a) $3$-regular tree, (b) labelling a path and colouring edges blue, (c) labelling a second path, (d) replacing a blue edge by two paths, (e) splitting labelled edges and appending paths.}
    \label{fig:TCG}
\end{figure}

\textit{Clique convergence of $T$:}
Since $T$ does not contain triangles, the second clique graph $k^2T$ is obtained by deleting all vertices of degree 1. These are precisely the end vertices of the attached paths. Note that this is equivalent to subtracting one from each edge-label, which is a symmetry operation of $T$. Thus, the graphs $T$ and $k^2T$ are isomorphic and $T$ is clique convergent.


\textit{A clique divergent quotient:}
Choose an $8$-cycle $C$ in $T$.
By construction $C$ contains precisely two vertices $v_1,v_2$ of degree four.
Let $T_i$ be the component of $T-E(C)$ that contains $v_i$.
Observe that $T_1$ and $T_2$ are isomorphic.
Hence, the fixed-point free symmetry of $C$ that rotates $v_1$ onto $v_2$ (and $v_2$ onto $v_1$) extends to a fixed-point free symmetry $\phi$ of $T$ that exchanges $T_1$ and $T_2$.
In other words, $T$ is a cover of $T':=T/\phi$.
Since there are still no triangles in $T'$, the $2n$-th clique graph $k^{2n}T'$ is obtained from $T'$ by decreasing all edge labels by $n$ (\ie\ shortening all attached paths by $n$).
However, since there is a distinguished 4-cycle, this is not a symmetry of $T'$. In fact, the distance of the 4-cycle from the edge with label one (eventually) increases with each application of the clique operator. Hence $T'$ is clique divergent.

\textit{A clique divergent cover:}
Start from a 16-cycle $C'$ and attach a copy of $T_1$ (from the last paragraph) to the vertices 0, 4, 8 and 12.
The fixed-point free symmetry of $C'$ that maps $0\to 8$ and $4\to12$ extends to a symmetry $\psi$ of $T''$.
Factoring out $\psi$ turns the 16-cycle into an 8-cycle, and we find $T\cong T''/\psi$.
As before, $k^{2n}T''$ decreases the edge labels by $n$, which is not a symmetry of $T''$ since it (eventually) increase the distance of the 16-cycle from the edge with label one. We conclude that $T''$ is clique divergent.



\section{Clique-Hellyness}\label{sec:cliquehelly}

One of the first classes of finite graphs shown to be clique convergent is the class of finite \emph{clique-Helly} graphs, see \cite{escalante1973iterierte}. 
A graph is called clique-Helly, if its set of cliques fulfils the \emph{Helly condition}, \ie\ if a set of cliques is pairwise intersecting, then it has a non-empty total intersection. 
A~vertex $v$ is said to \emph{dominate} a vertex $w$ if the closed neighbourhood of $w$ is a subset of the closed neighbourhood of $v$ (and note that $v$ and $w$ are necessarily adjacent).

The classical proof for convergence of clique-Helly graphs has the following main ingredients:

\begin{theorem}[\cite{escalante1973iterierte}, Satz 1]\label{thm:esc1}
    If $G$ is clique-Helly so is $kG$. 
\end{theorem}

\begin{theorem}[\cite{escalante1973iterierte}, Satz 2]\label{thm:esc2}
    For each clique-Helly graph $G$, $k^2G$ is isomorphic to an induced subgraph of $G$. This subgraph can be obtained by deleting all strictly dominated vertices and contracting each set of mutually dominating vertices into one vertex. 
\end{theorem}

Both proofs do not use finiteness and, thus, generalise directly to infinite graphs. While a chain of induced subgraphs of a finite graph stabilizes eventually with period length 1 or 2, this does not hold true for infinite graphs. The following examples provide infinite clique-Helly graphs of any finite period length as well as clique divergent ones.

\begin{example}
    It is easy to see that applying the clique operator twice to the graph $H_0$ in \cref{fig_retraction} is equivalent to deleting its vertices of degree $1$. But this deletion is equivalent to a right shift which is an isomorphism between $H_0$ and $k^2H_0$. By deleting the legs of $H_0$ in a periodic way, we can construct graphs of arbitrary even period length.

    \begin{figure}[htbp]
	\centering
	\quad\includegraphics[width=0.2\textwidth]{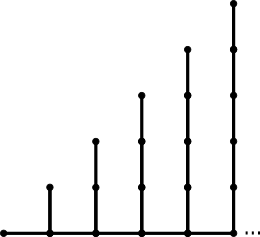}\qquad
	\caption{The graph $H_0$.} 
	\label{fig_retraction}
    \end{figure}

    Analogously, but slightly more complicated, it can be seen that for the graph $H_1$ from \cref{fig_retraction_3} (a) applying the clique graph operator once is equivalent to shifting right and flipping top and bottom.  By deleting the legs of $H_1$ in a periodic way, we can construct graphs of arbitrary period length. For example the graph $H_3$ in \cref{fig_retraction_3} (b) has period length $3$.

    From both $H_0$ and $H_1$ we can obtain clique divergent graphs by deleting legs in a non-periodic way.
\end{example}

   \begin{figure}[htbp]
	\centering
    (a)\includegraphics[width=0.25\textwidth]{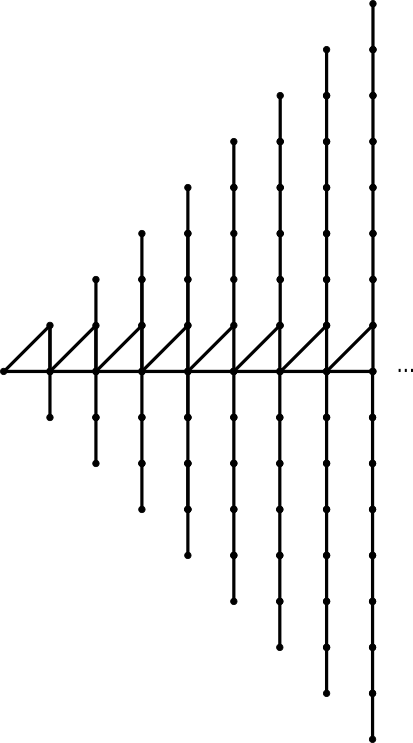}\qquad\qquad
	(b)\includegraphics[width=0.25\textwidth]{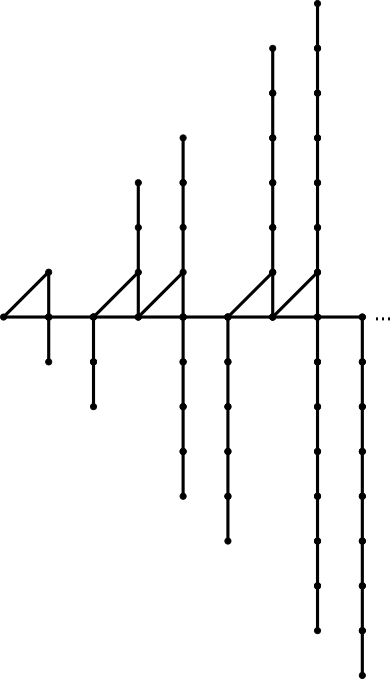}
	\caption{The graphs (a) $H_1$ and (b) $H_3$.}
	\label{fig_retraction_3}
\end{figure}

\section{Proof of \cref{thm:locgirth}}\label{sec:proofthmlocgirth}
We now prove the aforementioned \cref{thm:locgirth}:

\locgirth*

We proceed close to the proof of \cite[Proposition 18]{LARRION2002123}.
The following theorem is stated in \cite{LARRION2002123} for finite graphs, but the proof does not require finiteness.

\begin{theorem}[\cite{LARRION2002123}]
    If a graph $G$ has local girth at least 7, its clique graph $kG$ is clique-Helly.
\end{theorem}

Together with \cref{thm:esc1} and \cref{thm:esc2}, the following lemma implies that $kG\cong k^3G$ proving \cref{thm:locgirth}.

\begin{lemma}
     If a graph $G$ has local girth at least 7, and local minimal degree at least 2, no vertex of $kG$ dominates another vertex from $kG$. 
\end{lemma}

\begin{proof}
    We follow the line of arguments from \cite[Proposition 18]{LARRION2002123}.
    As $G$ has local girth at least $7$, it does not contain any $K_4$, and as the local minimal degree is at least $2$, each edge lies in at least two triangles, so all cliques of $G$ consist of exactly three vertices.
    Thus, let $\Delta$ and $\Delta'$ be two distinct triangles of $G$. If they do not intersect in at least one vertex, they are no neighbours in $kG$ and thus, they cannot dominate each other.
    Let's assume they intersect in an edge, say $\Delta=\{a,b,c\}$ and $\Delta'=\{a,b,c'\}$. As the local girth is at least $7$ and the local minimum degree is at least two, each edge in the induced neighbourhood of a vertex lies on at least one cycle of length at least 7. Thus, the neighbourhood of $c'$ contains at least a pair of adjacent vertices $d$ and $e$ apart from $a$ and $b$.
    The triangle $\Delta''=\{c',d,e\}$ intersects $\Delta'$ in $c'$ but $\Delta\cap \Delta''=\emptyset$, as otherwise $c\in \{d,e\}$ and  thus $\{a,b,c,c'\}$ would be a $K_4$.
    Now lets assume that $\Delta$ and $\Delta'$ intersect in exactly one vertex, say $\Delta=\{a,b,c\}$ and $\Delta'=\{a,b',c'\}$. Since each edge is contained in at least two triangles, there is a triangle 
    $\Delta''=\{a',b',c'\}$. Again, $\Delta\cap \Delta''=\emptyset$ as otherwise $a'\in\{b,c\}$ and $\{a, a', b',c'\}$ were a $K_4$.
\end{proof}

In light of the examples given in \cref{sec:cliquehelly}, and motivated from the study of local properties, we close with the following question:

\begin{question}
    If a class of graphs defined by local properties contains connected clique convergent graphs of arbitrarily large period, does it also contain a connected clique divergent graph?
\end{question}

We impose connectedness because otherwise the answer is trivially yes: the disjoint union of infinitely many clique convergent graphs with different periods is clique divergent. 

\bibliography{bibliography}

\begin{thebibliography}{1}

\bibitem{BAUMEISTER2022112873}
Markus Baumeister and Anna~M. Limbach.
\newblock Clique dynamics of locally cyclic graphs with $\delta\geq 6$.
\newblock {\em Discrete Mathematics}, 345(7):112873, 2022.

\bibitem{escalante1973iterierte}
Fernando Escalante.
\newblock {\"U}ber iterierte clique-graphen.
\newblock In {\em Abhandlungen aus dem Mathematischen Seminar der Universit{\"a}t Hamburg}, volume~39, pages 58--68. Springer, 1973.

\bibitem{LARRION2002123}
F.~Larri\'{o}n, V.~Neumann-Lara, and M.A. Pizaña.
\newblock Whitney triangulations, local girth and iterated clique graphs.
\newblock {\em Discrete Mathematics}, 258(1):123--135, 2002.

\bibitem{larrion2000locally}
Francisco Larri{\'o}n and Victor Neumann-Lara.
\newblock Locally {$C_6$} graphs are clique divergent.
\newblock {\em Discrete Mathematics}, 215(1-3):159--170, 2000.

\bibitem{LIMBACH2024114144}
Anna~M. Limbach and Martin Winter.
\newblock Characterising clique convergence for locally cyclic graphs of minimum degree $\delta\geq 6$.
\newblock {\em Discrete Mathematics}, 347(11):114144, 2024.

\end{thebibliography}
\bibliographystyle{plain}

\strictpagecheck
\checkoddpage
\ifoddpage
\newpage{\ }\thispagestyle{empty}
\fi

\end{document}